\documentclass{amsart}

\usepackage{amsmath,amssymb,amsthm,color,xcolor,mathrsfs,latexsym,tikz,tikz-cd,}
\usepackage{caption}
\usepackage{subcaption}
\usepackage[utf8]{inputenc}
\usepackage[all]{xy}
\usetikzlibrary{arrows,decorations.markings}
\usepackage[margin=1in]{geometry}
\usepackage{enumerate}
\usepackage{mathtools}
\usepackage{xfrac} 
\usepackage{nicefrac}
\usetikzlibrary{cd}
\usetikzlibrary{calc}
\usetikzlibrary{arrows.meta}
\usetikzlibrary{decorations.markings}
\usetikzlibrary{matrix,arrows}
\usepackage{lineno}
\usepackage{float}
\newcounter{casenum}

 \newcommand\restrict[1]{\raisebox{-.5ex}{$|$}_{#1}}
 
\newcommand{\Li}{\tilde{\mathcal{L}}^{(i)}}

\newcommand{\R}{\mathbb{R}}

\newcommand{\A}{\mathbb{A}}
\newcommand{\Q}{\mathbb{Q}}

\newcommand{\C}{\mathbb{C}}

\newcommand{\Ak}{\mathbb{A}_{\mathbb{C}}^k}
\newcommand{\X}{\mathcal{X}}
\newcommand{\x}{\tilde{\mathcal{X}}}
\newcommand{\blX}{\tilde{X}}

\newcommand{\PC}{\mathbb{P}_{\mathbb{C}}}

\DeclareMathOperator{\mult}{mult}

\theoremstyle{plain}
\newtheorem{thm}{Theorem}
\newtheorem*{thm*}{Theorem}
\newtheorem{lem}[thm]{Lemma}
\newtheorem{cor}[thm]{Corollary}
\newtheorem*{cor*}{Corollary}

\newtheorem{conj}[thm]{Conjecture}

\theoremstyle{definition}
\newtheorem{dfn}[thm]{Definition}

\newtheorem{rem}[thm]{Remark}

\makeatletter
\@addtoreset{proofpart}{theorem}
\makeatother

\theoremstyle{remark}

\usepackage{tikz}
\usepackage{todonotes}

\let\amsamp=&

\begingroup
\catcode`\&=13
\gdef\pampmatrix{%
  \begingroup
  \let&=\amsamp
  \begin{pmatrix}%
}
\gdef\endpampmatrix{\end{pmatrix}\endgroup}
\endgroup

\title{K\"ahler packings of projective complex manifolds.}
\author{Aeran Fleming}
\date{\today}

\begin{document}
\maketitle

\begin{abstract}
 
In this note we show that the multipoint Seshadri constant determines the maximum possible radii of embeddings of
K\"ahler balls and vice versa.
\end{abstract}

\section{Introduction.}
Seshadri constants were first introduced by Demailly in \cite{Dem} as a way of studying local positivity 
of ample line bundles at a given point of a variety. Since then, Seshadri constants have 
attracted substantial attention in the field of algebraic geometry and can be used to reformulate many classical ideas. A nice example of this is the famous conjecture of Nagata on plane algebraic curves, which in its original form is as follows:
\begin{conj}
Let $P_1,\ldots, P_k$ be points of $\mathbb{P}^2$ in general position and $m_1,\ldots,m_k$ be positive integers. Then for $k \geq 9 $ any curve $C \in \mathbb{P}^2$ of degree $d$ passing each point $P_i$ with multiplicity 
$m_i$ must satisfy 
\[ d \geq \frac{1}{\sqrt{k}} \sum_{i=1}^k m_i.\]

\end{conj}
Using multipoint Seshadri constants this can be rewritten as:
\begin{conj}
Let $P_1, \ldots, P_k$ be points of $\mathbb{P}^2$ in general position then for $k \geq 9$ the multipoint Seshadri constant 
\[\epsilon(\mathbb{P}^2,\mathcal{O}_{\mathbb{P}^2}(1),P_1,\ldots,P_k) = \frac{1}{\sqrt k}.\] 

\end{conj}

Motivated by work of Biran \cite{Biran}, who proved a symplectic analagon between Nagata's conjecture formulated as a packing problem, Eckl proved in \cite{Eckl} that Nagata's conjecture is in fact equivalent to a more 
restricted packing problem, namely a K\"ahler packing problem. More generally in dimension 2 there is a direct correspondence between sizes of multi-ball K\"ahler packings and multipoint Seshadri constants.  A similar result was obtained by David Witt-Nystrom in \cite{Witt} but this time for a variety of any dimension blown up at a single point.  The aim of this note is to generalise these results to projective complex manifolds of arbitrary dimension blown up at any number of points (in general position or not). During the writing of this paper a similar result was achived by Trusiani in \cite{Trusiani} however there are some differences in the formulation of the statement and its proof. 
In more details: 
\begin{dfn}
Let $(X,\omega)$ be an $n$ dimensional K\"ahler manifold with K\"ahler form $\omega$. Then a holomorphic embedding 
\[ \phi = \coprod_{q=1}^k \phi_q \colon \coprod_{q=1}^k B_0(r_q) \to X\]
is called a K\"ahler embedding of $k$ disjoint complex flat balls in $\C^n$ centered in $0$, of radius 
$r_q$, if there exists a K\"ahler form $\omega'$ such that $[\omega']=c_1(L)$ and 
 $\phi_q^*(\omega')= \omega_{std}$ is the standard K\"ahler form on $\C^n$ restricted to $B_0(r_q).
$ Let $\gamma_k(X,\omega;P_1,\ldots,P_n)=\sup\{ r > 0 \colon $ That there exists a K\"ahler packing 
with $ \phi_q(0)=P_q\}$. We call $\gamma_k$ the $k-$ball packing constant.  
\end{dfn}
Since curvature is an invariant under open holomorphic embeddings we cannot define the packing condition
$\phi^*_q(\omega)=\omega_{std}$ using the original K\"ahler form $\omega$ on $X$ as long as that
form is not flat enough around $q$. However the following theorem shows that under suitable conditions it is possible to find a K\"ahler form flat enough around $q$ in the cohomology class of $\omega$ as requested in the definition.

\begin{thm}\label{Main}
With the notation as above we have that
the square of the $k-$ball packing constant is equal to the  multipoint Seshadri constant:
\[ \gamma_k(X,\omega;P_1,\ldots,P_n) = \sqrt{\epsilon(X,L;P_1,\ldots,P_n)}.\]

\end{thm}

\section{Multipoint Seshadri constants.}

In this section we start by recalling some facts and definitions related to Seshadri constants. Single point Seshadri constants have been quite well studied and a good introduction to them can be found 
in Chapter 5 of \cite{Laz} and \cite{primer}. As we are interested in multipoint Seshadri constants we will introduce the notation and some ideas needed later. 

Let $X$ be a complex projective manifold of dimension $n$, with $P_1, 
\ldots, P_k$ distinct points of $X$. Let $L$ be an ample divisor and 
$\tilde{X} \coloneqq Bl_{P_1,\dots , P_k}(X) \xrightarrow{\pi} X $ be the blow-up of $X$ at 
the points $P_1, 
\ldots, P_k$. Let $\pi^{-1}(P_i) = E_i$ denote the exceptional divisor corresponding to the 
point $P_i$ for all $i=1,\ldots , k$ and for $\epsilon_i \in \Q_{>0}$ set 
\[ \tilde{L} \coloneqq \pi^*L-\sum_{i=1}^k \epsilon_i E_i \text{ for } \epsilon_i \in \mathbb{Q}_{>0} .\] If $L$ is a nef $\mathbb{Q}$-Cartier divisor then \cite{Laz} defines the multipoint Seshadri constant to be 
\[\max\{\epsilon \geq0 | \pi^*L-\sum \epsilon_iE_i \text{ is nef }\}.\] Since we require $L$ ample to achieve the desired K\"ahler packing we 
show that there is an equivalent definition where we replace $\max$ with $\sup$ and $\geq 0$ with $ >0$. The following lemma's prove that the two definitions are equivalent filling a gap in the literature. 

\begin{lem}\label{PAG}
 $\pi^*L-\epsilon'E$ is ample on $\tilde{X} \xrightarrow{\pi} X$ if
 $0<\epsilon' < \epsilon(X,L;P)$.
\end{lem}
\begin{proof}
By Seshadri's Criterion \cite[Thm 1.4.13]{Laz} if $L$ is ample then there exists $\epsilon_L > 0$ such that for all
points $Q \in X$ and all irreducible curves $C$ containg $Q$ we have that 
$\frac{L\cdot C}{\mult_QC} > \epsilon_L$. Let $\tilde{Q} \in \tilde{X}$ such that
 $\pi(\tilde{Q})=Q$ and let $\tilde{C} \subset \tilde{X}$ be an irreducible curve with 
 $\tilde{Q} \in \tilde{C}$. Then there are the following cases.
  If $Q \not\in E$ then \[\frac{(\pi^*L-\epsilon'E)\cdot \tilde{C}}
 {\mult_{\tilde{Q}}\tilde{C}}= \frac{L \cdot C}{\mult_QC}-\epsilon'\frac{\mult_PC}{\mult_QC} \geq
 \frac{\epsilon_L}{2}\text{ if } \frac{\mult_PC}{\mult_QC} <
  \frac{\epsilon_L}{2\epsilon'}.\] On the other hand
   if  $\frac{\mult_PC}{\mult_QC} \geq
  \frac{\epsilon_L}{2\epsilon'}$  then
  \[\frac{(\pi^*L-\epsilon'E)\cdot \tilde{C}}
 {\mult_{\tilde{Q}}\tilde{C}}=\frac{(\pi^*L-\epsilon_PE)\cdot \tilde{C}}
 {\mult_{\tilde{Q}}\tilde{C}}+(\epsilon_P-\epsilon')\frac{E\cdot \tilde{C}}{\mult_{\tilde{Q}}
 \tilde{C}} \geq (\epsilon_P-\epsilon')\frac{\mult_PC}{\mult_QC} \geq (\epsilon_P-\epsilon')
 \frac{\epsilon_L}{2\epsilon'}.\]If $\tilde{Q} \in E$ with $\tilde{C} \not\in E$ then $\mult_{\tilde{Q}}\tilde{C} \leq \mult_P \pi(\tilde{C})$ and 
 
  \[\frac{(\pi^*L-\epsilon'E)\cdot \tilde{C}}
 {\mult_{\tilde{Q}}\tilde{C}} \geq \frac{(\pi^*L-\epsilon'E)\cdot \tilde{C}}
 {\mult_{P}\pi(\tilde{C})} \geq \epsilon_P -\epsilon' > 0.\] 
Finally if $\tilde{Q} \in E$ and $\tilde{C} \subset E$ then 
  \[ \frac{(\pi^*L-\epsilon'E)\cdot \tilde{C}}
 {\mult_{\tilde{Q}}\tilde{C}}\geq - \frac{\epsilon'E\cdot\tilde{C}}{\mult_{\tilde{Q}}\tilde{C}} \geq 1.\]
  The final inequality appears since $E\cdot \tilde{C} = \deg\tilde{C}$ on $E \cong \mathbb{P}^{n-1}$
 and $\deg\tilde{C} \geq \mult_{\tilde{Q}} \tilde{C}$. Hence we have shown that if $0 < \epsilon' < \epsilon_P$ then $\pi^*L-\epsilon'E$ is $\mathbb{Q}$-ample.

\end{proof}
\smallskip
\begin{cor}
$\max\{\epsilon \geq 0 :\pi^*L-\epsilon E$ is nef $\}=\sup\{\epsilon >0 : \pi^*L-\epsilon E$ is 
$\mathbb{Q}$-ample$\}$.\qed
\end{cor} 

\begin{lem}\label{ample}
There exists $\epsilon_1,\ldots,\epsilon_k > 0 $ such that $\tilde{L}$ is $\Q$-ample for all 
$i=1,\ldots,k$. 
\end{lem}
\begin{proof}
By Seshadri's criterion and Lemma \ref{PAG}
we know there exist some $\epsilon_1 > 0$ such that $\pi^*L-\epsilon_1 E$ is $\mathbb{Q}$-ample
on $\tilde{X}_1=Bl_{P_1}(X)$. Choosing a second point $P_2 \in X$  Seshadri's criterion again ensures the existence of some $\epsilon_2 > 0$
such that $\pi^*L-\epsilon_1E_1-\epsilon_2E_2$ is $\mathbb{Q}$-ample on $\tilde{X}_2=Bl_{P_1,P_2}(X)$. Arguing iteratively for 
$k$ points of $X$ we obtain the claim of the lemma. 
\end{proof}

From the statement of the above lemma it is not obvious that we can chose all the $\epsilon_i$ to be 
equal. This follows from the next lemma.

\begin{lem}\label{lem2}
If $\pi^*L-\sum_{i=1}^k \epsilon_iE_i$ is $\mathbb{Q}$-ample then $\pi^*L-\sum_{i=1}^k \epsilon'_iE_i$
is ample if $ 0<\epsilon'_i\leq \epsilon_i$, for all $i=1,\ldots,k$.
\end{lem}
\begin{proof}
 Lemma \ref{PAG} proves the case when $k=1$. If we show that $\pi^*L-\sum_{i=1}^k \epsilon_i E_1$ ample implies $\pi_{k-1}^*L-\sum_{i=1}^{k-1}\epsilon_i E_1$ is ample, where $\pi_{k-1}
 \colon \tilde{X} \to X$ denotes the blow up of $X$ at the first $k-1$ points, then  we can simply apply Lemma \ref{PAG} on the last blow up. Let $C \subset X$ be an irreducible curve and 
 $\bar{C} \subset \tilde{X}$ be the strict transform of $C$. Then
 \begin{align*}
 \frac{(\pi^*L-\sum_{i=1}^{k-1} \epsilon_i E_1) \cdot C}{\mult_QC}
 &=\frac{(\pi_{k-1}^*L-\sum_{i=1}^{k-1} \epsilon_i E_1)\cdot \bar{C}}{\mult_Q C}\\
 &= \frac{(\pi^*L-\sum_{i=1}^{k} \epsilon_i E_1)}{\mult_Q C} + \frac{\epsilon_iE_i \bar{C}}
 {\mult_QC} \\
 &= \frac{(\pi^*L-\sum_{i=1}^{k} \epsilon_i E_1)}{\mult_Q C}+\epsilon_i \cdot \frac{\mult_{P_k}C}
 {\mult_QC}\\
 &\geq \epsilon_{\pi^*L-\sum_{i=1}^k \epsilon_iE_i}.
 \end{align*} 
 Thus Seshadri's Criterion implies the ampleness of $\pi_{k-1}^*L-\sum_{i=1}^{k-1}\epsilon_i E_i$.

\end{proof}

\begin{dfn}\label{seshadri}
For $X$ and $L$ defined as above the multipoint Seshadri constant associated to points $P_1,\ldots P_k$
of X is 
\[ \epsilon(X,L;P_1,\ldots,P_k)=\text{ sup}\{\epsilon\in\mathbb{Q}_{>0} \colon \pi^*L-\sum_{i=1}^k
\epsilon_i
E_i \text{ is } \mathbb{Q}\text{-ample}\}.\]
\end{dfn}\noindent The multipoint Seshadri constant associated to an ample divisor is always $ > 0$ by Lemma \ref{lem2}.

 \medskip
 \section{Degeneration of complex projective manifolds to multipoint blow up.} 
 First we fix some general notation and introduce some constructions that will be 
used later. The notation fixed in this section will be used for the remainder of the report unless 
otherwise stated.

Let $X$ and $L$ be defined as above and consider the product manifold 
\begin{figure}[H]
\begin{center}
\begin{tikzcd}
\mathcal{X}\coloneqq X \times \Ak  \arrow{r}{p} \arrow{d}{q} & \X \\ \Ak
\end{tikzcd}
\end{center}
\label{Family}
\end{figure}
\noindent Then if $t_1,\ldots,t_k$ represent coordinates of $\Ak$ and set $Z_i \coloneqq \{P_i\} \times \{t_i=0\}$. These are exactly the 
coordinate hyperplanes cut out by $t_i=0$ over the points $P_i$. 
\begin{rem}
For all $1 \leq i < j \leq k$ we have $Z_i \cap 
Z_j = \emptyset$ since the points $P_i
\in X$ are distinct.
\end{rem}
By blowing up the family $\X$ over the union of all the centres $Z_i$ we obtain a new algebraic 
family over $\mathbb{A}^k_{\mathbb{C}}$.
\begin{figure}[H]
\begin{center}
\begin{tikzcd}
 \x \coloneqq Bl_{\cup_{i=1}^k Z_i}(\mathcal{X}) \arrow{r}{\Pi} \arrow[swap]{dr}{(q\circ \Pi)} & \X 
 \arrow{d}{q} \arrow{r}{p} & X \\
     & \Ak
\end{tikzcd}
\end{center}
\label{Fam}
\end{figure}
\noindent such that the exceptional divisor corresponding to the centre $Z_i$ is  
$ \mathcal{E}_i \coloneqq \Pi^{-1}(Z_i)$.
Set $\tilde{\mathcal{L}}_{dm_i,\ldots,m_k}=d\Pi^*p^*_{1}L-\sum_{i=1}^k
m_i\mathcal{E}_i $ a divisor on $\x$. 
\begin{rem}
\begin{enumerate}
\item $\x$ is a well defined complex manifold projective over $\Ak$.
\item$\mathcal{E}_i \cong Z_i \times \PC^n$
\item $(q \circ \Pi)^{-1}\big((t_1,\ldots ,t_k)\big)$ is the blowup of $X$ in the points $P_i$ where $t_i=0$.
\end{enumerate}
\end{rem}
Running through an appropriate path through the 
parameter space $\Ak$ these blow ups accumulate iteratively. Let $\delta_1,\ldots ,\delta_k $ be positive real numbers
and define lines in $\Ak$,
\[l_i \coloneqq \{(0,\ldots,0,t\cdot \delta_i,\delta_{i+1},\ldots,\delta_k) \colon t \in \R\}
\] consisting of points whose first $(i-1)$ coordinates are zero, the $i$-th coordinate is of the form $t\cdot 
\delta_i$ and the remaining $(i+1)$ coordinates are constant. For the point $t^{(i)}=
(0,\ldots,0,\delta_i,\delta_{i+1},\ldots,\delta_k) \in l_i$ for $i=1,\ldots, k$, the preimage is 
\[(q\cdot \Pi)^{-1}(t^{(i)})=Bl_{P_1,\ldots,P_{i-1}}(X) \cup \bigcup_{j=1}^{i-1}\mathcal{E}_j \cap
(q\cdot \Pi)^{-1}(t^{(j)}).\]
Following the line from $l_i$ from $t^{(i)}$ to $t^{(i+1)}$ we trace a path through the parameter space such that the preimage of $t^{(1)}$ is simply $X$, the preimage of $t^{(1)}$ is simply the blow up of $X$ at $P_1$ with some contributions from the exceptional divsior and so on until we get that the preimage of $t^{(k)}=(0,\ldots,0)$ is 
the blow up of $X$ at $P_k$ with contributions from all exceptional divisors. See Figure \ref{parameter}
for a visualisation.

\begin{figure}[h]
\centering     
\begin{subfigure}[t]{0.45\textwidth}         
 \centering          
 \resizebox{\linewidth}{!}{             
  \begin{tikzpicture} 
  [scale=1.6]
\fill[blue] (2,0) circle [radius=0.4mm];
\node at (2.1,-0.1) {$0$};
\draw [->](2,0) to (3.5,0);
\draw [->](2,0) to (2,1.5);

\draw[dashed] (1,0.5)--(4,0.5);
\draw[red,thick] (2,0.5) to(3,0.5);
\draw[blue,thick] (2,0.5) to (2,0);
\fill[red](2,0.5)circle[radius=0.4mm];
\node at (1.8,0.65) {$t^{(1)}$} ;
\fill(3,0.5)circle [radius=0.4mm];
\node at (3.1,0.65) {$t^{(0)}$};
\node at (0.7,0.6){$l_1$};
\node at (2,1.7) {$t_2$};
\node at (3.75,0) {$t_1$};
\end{tikzpicture}          }          
\caption{Parameter space for k=2}        
  \end{subfigure}     
  \begin{subfigure}[t]{0.45\textwidth}      
  \centering          
  \resizebox{\linewidth}{!}{              
  \begin{tikzpicture}   
[scale=1.7]
\fill [green](2,0) circle [radius=0.4mm];
\draw [->](2,0) to (3.25,1.25);
\draw [->](2,0) to (3.5,0);
\draw [->](2,0) to (2,1.5);

\draw[red,thick] (2.25,0.5) to (3,0.5);
\fill[red](2.25,0.5) circle [radius=0.4mm];
\node at (2.3,0.65) {$t^{(1)}$};

\draw[blue,thick] (2,0.25) to (2.25,0.5);
\draw[dashed] (1.75,0) to (3,1.25);
\node at (3,1.35) {$l_2$};
\fill[blue](2,0.25) circle [radius=0.4mm];
\node at(1.8,0.3) {$t^{(2)}$};
\draw[green,thick] (2,0.25) to(2,0);
\draw[dashed] (1,0.5)--(4,0.5);
\fill(3,0.5)circle [radius=0.4mm];
\node at(3.1,0.65) {$t^{(0)}$};
\node at (2.1,-0.1) {$0$};
\node at (4,0) {$t_1$};
\node at (2,1.7) {$t_2$};

     \end{tikzpicture}          }          
  \caption{Parameter space for k=3}               
  \end{subfigure}      
  \caption{}   
  \label{parameter}  
  \end{figure}

Similarly for $l_i \cong \A_{\C}^1$ we have 
\[ (q\cdot \Pi)^{-1}(l_i) = \Big( \bigcup_{j=1}^{i-1}\mathcal{E}_j \cap (q\cdot \Pi)^{-1}(l_i)\Big)
\cup \x_i.\]
Where $\mathcal{E}^{(i)}=\mathbb{P}^n_{\C}$ and $\x_i$ is the blow up of $\blX_{i-1} \times l_i$ in the
point $(P_i,0)$.
\begin{figure}[H]
\begin{center}
\begin{tikzcd}
 \x_i \arrow{r}{\Pi_i}\ar[bend left=30]{rr}{\Pi^{(i)}} & \X_i = \blX_{i-1} \times l_i \arrow{d}{p_{i-1}} \arrow{r}{q_{i-1}} & l_i \cong \A^1_{\C} \\
     & \blX_{i-1}

\end{tikzcd}
\end{center}
\label{Family x_i}
\end{figure}
On $\x_i$ there exists a family of divisors
\[ \tilde{\mathcal{L}}^{(i)}_{d,m_1,\ldots,m_i} \coloneqq  d\Pi^{(i) *} p_{i-1}^*
\tilde{L}^{i-1}_{d;m_1,\ldots,m_i}-m_i \mathcal{E}_i^{(i)}\]
\noindent where $\mathcal{E}_i^{(i)} \cong \mathbb{P}^n_{\C}$ is the exceptional divisor of 
the blow up $\Pi_i$.
Specific divisors associated to a particular fiber are denoted $\mathcal{L}^{(i)}_{d,\underline{m},t}$
, 
where $t$ denotes the parameter on $\A^1_{\C}$ and $d,\underline{m}_i=(m_1,\ldots,m_i)$ record the degree and multiplicity (we will often just write $\Li_t$ for short if the multiplicity and degree are fixed).
\begin{rem}
$ \x_{i,0}= \tilde{X}_i \cup \mathcal{E}_i^{(i)}$, where $ \mathcal{E}_i^{(i)} \cong \PC^n$ and
$\x_{i,t}=\Pi^{(i)\ -1}(t)$ is the fiber over $l_i\cong \A^1_{\C}$. 
Hence there exists positive integers $d, m_j$ such that   
 $\tilde{L}^{(i)}=\tilde{L}^{(i)}_{d;\underline{m}_i} \coloneqq d\pi^*L-\sum_{j=1}^i m_jE_j$ is an ample divisor on $\blX_i$ (and very ample for $d,m_j\gg 0$), if $\frac{m_j}{d} = \epsilon_j$ for all $\epsilon_j$ as in 
Lemma \ref{ample}.
\end{rem}

Now that we have defined a algebraic family $\x_i$ and a family of divisors  
$\tilde{\mathcal{L}}^{(i)}=\tilde{\mathcal{L}}^{(i)}_
{d;\underline{m}_i}$  we would like to investigate the global sections of these divisors restricted to the fibers of $\Pi^{(i)}$, and how these global sections are constructed for different $i$. This is shown by the following 
Theorem:
\begin{thm} \label{thmA}
For all $i=1,\ldots,k$ there exists global sections $\sigma_0^{(i)},\ldots, \sigma_{N_i}
^{(i)}$ of $\Pi^{(i)}_{*}\tilde{\mathcal{L}}^{(i)} $ such that:
\begin{enumerate}
\item $\sigma_0^{(i-1)},\ldots, \sigma_{N_i}
^{(i-1)}$ trivialise $\Pi^{(i)}_{*}\tilde{\mathcal{L}}^{(i)} $.\\
\item If $0_i$ denotes the zero of the line $l_i$ then the sections \newline 
\[\sigma^{(i)}_{0,0_i}\restrict{\mathcal{E}_i^{(i)}},\ldots, \sigma^{(i)}_{{N_i},0_i}\restrict{\mathcal{E}
_i^{(i)}} \]
generate the K\"ahler form $m_i\cdot \omega_{FS}$ on $\mathcal{E}_i
^{(i)}
 \cong \PC^n$.\\
 \item The restricted sections $\sigma_{0,\delta_i}^{(i)},\ldots, \sigma_{N_i,\delta_i}
^{(i)}$ generate the same K\"ahler form on $\x_{i,\delta_i} \cong \blX_{i-1}$ as;\newline
$\sigma^{(i-1)}_{0,0_{i-1}}\restrict{\blX_{i-1}},\ldots, \sigma^{(i-1)}_{N_{i-1},0_{i-1}}\restrict{\blX_{i-1}}$ 
  from the previous family $\x_{i-1}$.
\end{enumerate}

\end{thm}
\begin{proof}
If we show that the dimension of the space of global sections of $\tilde{\mathcal{L}}^{(i)}$ restricted 
to each fiber of $\Pi^{(i)}$ is the same,
Grauert's semi-continuity theorem \cite[III.12]{Har} implies immediately that $\Pi_*^{(i)}\tilde{\mathcal{L}}^{(i)}$ is locally 
free, hence it is free by the Quillen-Suslin Theorem. Thus there exists trivialising global sections of $\Pi^*\tilde{\mathcal{L}}^{(i)}$. But
instead of using these powerful theorems we describe the trivialising sections directly because they allow us to prove the properties of the theorem more easily. For the sections to be trivialsing it is enough to show that they form a basis of $H^0(\tilde{\mathcal{X}}_{i,t},
\tilde{\mathcal{L}}^{(i)}_t)$ then restricted to $\tilde{\mathcal{X}}_{i,t}$ for all $t \in \mathbb{A}^1_{\mathbb{C}}$.

  We start by calculating a basis of sections of $H^0(X,\mathcal{O}_X(dL))$ characterised by their vanishing behaviour at the points $P_1,\ldots,P_k$ so that subsets of this basis can be interpreted as generating sections of
   $H^0(\blX,\tilde{L}^{(i)})$. 
 Let $\mathfrak{m}_{X,P_i}$ denote the maximal ideal of $X$ at $P_i$ and consider the short exact sequence
\[ 0 \to \bigcap^k_{i=1}\mathfrak{m}_{X,P_i}^{m_i+1} \otimes \mathcal{O}_X(dL)\to \mathcal{O}_X(dL) 
\to \bigoplus_{i=1}^k \nicefrac
{\mathcal{O}_X}{\mathfrak{m}^{m_i+1}_{X_i,P_i}} \to 0.\]
Taking the long exact sequence of cohomology w.r.t to the above sequence gives
\begin{figure}[H] 
\centering
\begin{tikzpicture}[descr/.style={fill=white,inner sep=1.5pt}]
        \matrix (m) [
            matrix of math nodes,
            row sep=2em,
            column sep=4em,
            text height=1.5ex, text depth=0.25ex
        ]
        { 0 & H^0(X,\bigcap^k_{i=1}\mathfrak{m}_{X,P_i}^{m_i+1} \otimes \mathcal{O}_X(dL)) & 
        H^0(X,\mathcal \mathcal{O}_X(dL) )
         & H^0(X,\mathcal \bigoplus_{i=1}^k \nicefrac
         {\mathcal{O}_X}{\mathfrak{m}^{m_i+1}_{X,P_i}}) \\
            & H^1(X,\bigcap^k_{i=1}\mathfrak{m}_{X,P_i}^{m_i+1} \otimes \mathcal{O}_X(dL)) & H^1(X,\mathcal{O}_X(dL))) & \ldots \ldots \ldots \ldots \ldots \\
        };

        \path[overlay,->, font=\scriptsize,>=latex]
        (m-1-1) edge (m-1-2)
        (m-1-2) edge (m-1-3)
        (m-1-3) edge (m-1-4)
        (m-1-4) edge[out=355,in=175] node[descr,yshift=0.3ex]{ } (m-2-2)
        (m-2-2) edge (m-2-3)
        (m-2-3) edge (m-2-4);
  
\end{tikzpicture}
\end{figure}
By Serre Vanishing $H^1(X,\bigcap^k_{i=1}\mathfrak{m}_{X,P_i}^{m_i+1} \otimes \mathcal{O}_X(dL))=0$, since applying the projection formula yields
\[ H^1(X,\bigcap^k_{i=1}\mathfrak{m}_{X,P_i}^{m_i+1} \otimes \mathcal{O}_X(dL))= H^1(\blX , 
\mathcal{O}_X(\tilde{L}_{d,m_i})).\] 
and we assume $\tilde{L}_{d,m_i}$ is ample and $d,\underline{m}_k \gg 0$. Hence $\phi 
\colon H^0(X,\mathcal{O}_X(dL)) \to \bigoplus_{i=1}^k
 \sfrac{\mathcal{O}_X}{\mathfrak{m}_{X,P_i}^{m_i+1}}$
 is surjective. This gives a basis of sections of $H^0(X,\mathcal{O}_X(dL))$ which is the union of:
 \begin{itemize}
 \item A basis $B_0$
  of $H^0(X,\bigcap_{i=1}^k \mathfrak{m}^{m_i+1} \otimes \mathcal{O}_X(dL))$, which are sections of $\mathcal{O}
  _X(dL)$ vanishing to multiplicity at least $m_i+1$ in $P_i $ for $i=1,\ldots,k$.
  \item A set $B_i $ of sections  of $\mathcal{O}_X(dL)$ mapped to $0$ in $ \sfrac{\mathcal{O}_X}{\mathfrak{m}
  _{X,P_j}^{m_j+1}}$ for $j \neq i$ and to a basis of homogenous polynomials of degree $\leq m_i$ in 
  $ \sfrac{\mathcal{O}_X}{\mathfrak{m}_{X,P_i}^{m_i+1}}$, in variables given by local coordinates around $P_i$. 
  \item A set of sections $\tilde{B}_i$ of $\mathcal{O}_X(dL)$ as above, but with  homogenous polynomials 
  of degree $=m_i$, and that $\tilde{B}_i \subset B_i$.

 \end{itemize}
 Hence a basis of $H^0(\blX_i,\mathcal{O}_{\blX_i}(\tilde{L}^{(i)}))$ is given by
 \[ B^{(i)}\coloneqq B_0 \cup \bigcup_{j=1}^{i}\tilde{B}_j \cup \bigcup_{j=i+1}^kB_j.\] 
To change between a basis of sections on $X_{i-1}$ and on $X_{i}$ we simply skip those sections which vanish to order 
less than $m_i$  at $P_i$. 

Since $\x_i= \blX_{i-1}\times \A_{\C}^1$ we can use these sections to construct the sections of $H^0(\x_i,\tilde{\mathcal{L}}^{(i)})$ 
using the following procedure:
\begin{enumerate}[Step 1:]
  \item Pull back a section $s \in H^0(\tilde{X}_{i-1},\mathcal{O}_{X_{i-1}}(\tilde{L}^{(i-1)}))$ 
 along $p_{i-1}$  
 to get a section of $p_{i-1}^*\tilde{L}^{(i-1)}=\mathcal{L}^{(i-1)}$
   on $\X_i=\blX_{i-1} \times \A_{\C}^1$.
 \item If  mult$_{P_{i}}s < m_i$ then the section $t^{m_i-\mbox{ mult }_{P_{i}}s}p_{i-1}^*s 
\in H^0(\x_i, \tilde{\mathcal{L}}^{(i-1)})$ has multiplicity 
 $m_i$ in $(P_i,0)$. If mult$_{P_i}s \geq m_i$ then mult$_{(P_i,0)}p_{i-1}^*s \geq m_i$.
 \item In both cases, we can subtract $m_i$ copies of the exceptional divisor $\mathcal{E}_i$ from the 
 pullback of this section along $\Pi_i$, thus obtaining sections of 
 $H^0(\x_{i},\tilde{\mathcal{L}}^{(i)})$. 
\end{enumerate}

When starting with the basis sections  in $B^{(i-1)}$ , restricting to the general fibers 
 $\x_{i,t} \cong \blX_{i-1}$ of $\Pi^{(i)}$ (that is 
$t\neq0$) we get back the sections 
in $B^{(i-1)}$ possibly multiplied with some power of $t$. Thus all the sections $ \sigma_1^{(i-1)}, \ldots \sigma_{N_i}
^{(i-1)}$
obtained from $B^{(i-1)}$ in the way described above are trivializing $\Pi^{(i)}_*\Li$ outside the central fiber. 
To understand the restriction of these sections  to the central fiber 
 we need to analyse 
$\tilde{\mathcal{L}}^{(i)}\restrict{\x_{i,0}}$ and its global sections. 
Note, $\x_{i,0}=\blX_i \cup \mathcal{E}_i^{(i)}$ and $\blX_i \cap \mathcal{E}_i^{(i)}=E_i$. 
Furthermore $\tilde{\mathcal{L}}^{(i)} \restrict{\blX_{i}}=\tilde{L}^{(i)}$, and  
$\tilde{\mathcal{L}}^{(i)} \restrict{\mathcal{E}_{i}^{(i)}}=\mathcal{O}_{\mathcal{E}_i^{(i)}}(-m_i\mathcal{E}_i^{(i)})
\cong \mathcal{O}_{\mathbb{P}^n}(m_i)$ via the 
isomorphism $\mathcal{E}_i^{(i)}\cong \mathbb{P}^n_{\C}$. This means that global sections of $\tilde{
\mathcal{L}}^{(i)}\restrict{\x_{i,0}}$ consist of a section of
$\mathcal{L}^{(i)}\restrict{\blX_{i}}$ and a section of $\mathcal{L}^{(i)}\restrict{\mathcal{E}_i^{(i)}}$ 
coinciding when restricted to $E_i$. 
\begin{itemize}
    \item Global sections of $\tilde{\mathcal{L}}^{(i)}\restrict{\blX_i}$ and their restriction to $E_i$:\newline
      These are all sections of $dL$ on $X$
       which vanish to multiplicity $\geq m_j$ in $P_j$ for $j=1,\ldots,i.$ Equivalently 
       these are global sections of $\tilde{L}^{(i-1)}$ on $\blX_{i-1}$ vanishing in multiplicity $\geq 
       m_i$ in $P_i$. To obtain a section of $\tilde{L}^{(i)}$ from a section of $ \tilde{L}^{(i-1)}$
       we pull back the section along the map $\pi^{(i-1)} \colon \blX_i \to \blX_{i-1}$ then subtract $
       m_i$ copies of the exceptional divisor. 
        If a section of $\Li 
  \restrict{\blX_i} = \tilde{L}^{(i)}$ corresponds to a section of $\tilde{L}^{(i-1)}$ with multiplicity greater than $m_i$ at 
  $P_i$ the restriction of this section to $E_i$ is zero. If the multiplicty is exactly $m_i$ at $P_i$ then   
   $\tilde{\mathcal{L}}^{(i)}\restrict{E_i}\cong\mathcal{O}_{\mathbb{P}^{n-1}}(m_i)$ hence the restriction
    of such a section of $\Li$ to $E_i \cong \mathbb{P}^n_{\C}$
    will be described by a non-zero homogeneous  polynomials of degree $m_i$ in $Y_1,\ldots,Y_n$. 
    \item Global sections of $\tilde{\mathcal{L}}^{(i)}\restrict{\mathcal{E}_i^{(i)}}$
    and their restriction to $E_i$:\newline
 To describe these we introduce homogeneous coordinates $[T:Y_1:\ldots:Y_n] $ on $\mathcal{E}_i^{(i)}$. The 
 coordinates $Y_1,\ldots,Y_n$ come from local coordinates $y_1,\ldots,y_n$ around $P_i$,
  the $T$ coordinate comes from the affine base parameter $t$, and
    $T=0$ describes $E_i \subset \mathcal{E}_i^{(i)}$. Hence the restriction of a section of 
    $\tilde{\mathcal{L}}^{(i)}\restrict{\mathcal{E}_i^{(i)}}$ to $E_i$ is obtained by setting $T=0$. 
 Since $\tilde{\mathcal{L}}^{(i)}\restrict{\mathcal{E}_i^{(i)}} \cong \mathcal{O}_{\mathbb{P}^n}(m_i)$ sections 
 of $\tilde{\mathcal{L}}^{(i)}\restrict{\mathcal{E}_i^{(i)}}$ are non-zero 
  homogeneous polynomials of degree $m_i$ in $T,Y_1,\ldots,Y_n$.

    \end{itemize}
     Sections of $\Li$ are therefore described by  pairs of sections of $ \blX_i$  and 
     $\mathcal{E}_i^{(i)}$  both vanishing on 
   $ E_i$ and pairs of sections of  $\blX_i$  and  $\mathcal{E}_i^{(i)}$ 
    restricting to the same non-zero homogeneous polynomial in $ Y_i,\ldots,Y_n$
     of degree $ m_i $  on $ E_i.$
   A basis  of global sections of $\Li\restrict{\x_{i,0}}$ can therefore be built from the basis of $H^0(\tilde{X}_i, L^{(i)})$ 
and $H^0(\tilde{\mathcal{E}}_i^{(i)}, \tilde{\mathcal{L}}^{(i)}\restrict{\mathcal{E}_i^{(i)}}$ as the set of
    \begin{itemize}
    \item  Pairs of basis sections on $\blX_i$ vanishing on $E_i$ and the zero section on
    $\mathcal{E}_i^{(i)}$,\\
     \item Pairs of the zero section on $\blX_i$ and a basis section on $\mathcal{E}_i^{(i)}$ vanishing on $E_i$,\\
    \item  Pairs of basis sections of  $\blX_i$ and $\mathcal{E}_i^{(i)}$ that restrict to the same 
   non-zero  homogeneous polynomial of degree $m_i$ on $E_i$.  
    \end{itemize}
  
Now we want to show that the restriction of the $\sigma^{(i)}_0 \ldots \sigma^{(i)}_{N_i}$ to $\x_{i,0}$ yields a basis 
of $\Li \restrict{\x_{i,0}}$. To this purpose, we follow steps 1-3 to construct $\sigma _j^{(i)}$ from a section of 
$\tilde{L}_{i-1}$ and then identify the pair of sections describing the restriction to $\x_{i,0}$. 
If a section of $\tilde{L}_{i-1}$ vanishes to multiplicity strictly greater than $m_i$ at $p_i$ then it corresponds to a pair 
of sections on $\Li_{0}$ consisting of a 
section of $\Li$ that restricts to zero on $E_i$, and the zero section on $\mathcal{E}_{i}^{(i)}$. If the section of $
\tilde{L}_{i-1}$ on $\blX_{i-1}$ vanishes to exactly multiplicity $m_i$ at $P_i$ then pulling back and subtracting $m_i$ 
copies of the exceptional divisor we obtain non-zero sections
on $\blX_i$ and on $\mathcal{E}^{(i)}_i$. Such a section of $\tilde{L}_{i-1}$ corresponds 
to a pair of sections on $\x_{i,0}$ that restrict to the same non-zero homogeneous monomial on $E_i$. Finally if 
a section$ s$ of $\tilde{L}_{i-1}$ on $\blX_{i-1}$ vanishes to  multiplicity  strictly less than $m_i$ at $P_i$ then after 
pulling back along $p_{i-1}$ we must multiply by $t^{m_i-\mbox{ mult }_{P_{i}}s}$ before we can subtract copies of 
the exceptional divisor. This type of section corresponds to a pair of sections on $\x_{i,0}$ consisting of the zero 
section on $\blX_{i}$
and a section on $\mathcal{E}_{i}^{(i)}$ that restricts to zero on $E_i$. 

 In terms of the sets $B_0, B_j$ and $\tilde{B}_j$, making up the basis $B^{(i)}$ of $\tilde{L}^{(i)}$ we find that 
a basis of global sections of $\tilde{L}^{(i-1)}$ vanishing to 
multiplicity $> m_i$ at $P_i$ can be written as $B_0 \cup \bigcup^{i-1}_{j=1}\tilde{B}_j \cup \bigcup_{i+1}^k B_j$. 
A basis of global sections vanishing to multiplicity exactly $m_i$ at $P_i$ is $\tilde{B}_i$ and finally a basis of global  
sections vanishing to multiplicity 
less than $m_i$ at $P_i$ is given by $B_i - \tilde{B}_i$. The union of all three basis provides a basis $B^{(i-1)}$ for 
$\x_{i,t} \cong \blX_{i-1}$. The union of the basis corresponding to all sections of $\tilde{L}^{(i-1)}
$ vanishing to multiplicity less than or equal to $m_i$ at $P_i$  is isomorphic to a basis of global sections of $\Li 
\restrict{\mathcal{E}^{(i)}_i}$ where the  identification is given by the 
correspondence between homogeneous polynomials in local coordinates around $P_i$ and homogeneous coordinates 
of $\mathcal{E}^{(i)}_i \cong \mathbb{P}^n_{\C}$.

Thus, the sections $\sigma_0^{(i)},\ldots, \sigma_{N_i}
^{(i)}$  also restrict to a basis of $\Li \restrict{\x_{i,0}}$, hence trivialize $\Pi^{(i)}_*\Li$ over all fibers of the family. 

Now let us construct a K\"ahler form using the trivializing sections of $\Li$ restricted to 
$\mathcal{E}_i^{(i)} \cong \mathbb{P}^n_{\C}$. On this exceptional divisor we can choose homogeneous coordinates 
$T,Y_1,\ldots Y_n$ as above. Then 
\begin{align*}
m_i \cdot\omega_{FS}&=\frac{i}{2\pi}\partial\bar{\partial}m_i\cdot \log(Y_1\bar{Y_1}+
\ldots+Y_n\bar{Y}_n+T\bar{T}) \\ 
       &=\frac{i}{2\pi}\partial\bar{\partial} \log(Y_1\bar{Y}_1+
\ldots+Y_n\bar{Y}_n+T\bar{T})^{m_i} \\
       &=\frac{i}{2\pi}\partial\bar{\partial}\log(Y_1^{m_i}\bar{Y_1}^{m_i}+ m_iY_1^{m_i-1}\bar
       {Y}_1^{m_i-1}Y_2\bar{Y}_2
\ldots) \\ 
       &= m_i\cdot \omega_{FS} = \frac{i}{2\pi}\partial\bar{\partial}\log( \sum_{\alpha+\beta=m_i}
       c_{\alpha ,\beta}Y^{\alpha}
\bar{Y}^{\alpha}T^{\beta}\bar{T}^{\beta}),\  \text{ where } c_{a,b} \text{ are positive integers. }
\end{align*}
If we choose the sections $\sigma_j^{(i)}$ that do not vanish on the exceptional divisor (i.e. coming from sections of $
\tilde{L}^{(i-1)}$ vanishing with multiplicity $\leq m_i$ in $P_i$) such that they restrict to the basis of monomials 
$\sqrt{c_{\alpha,\beta}}Y^{\alpha}T^{\beta}$ of the homogeneous polynomials of degree $m_i$, then $m_i \cdot
\omega_{FS}$ is the K\"ahler form on $\mathcal{E}^{(i)}$ generated by the restriction of these sections $\sigma_{j}
^{(i)}$. Note that $c_{\alpha,\beta}$ is a positive integer, so $\sqrt{c_{\alpha,\beta}}$ is just the usual real square 
root. 

To achieve property (3) we have to construct the trivializing sections of $\Li$ on $\x_i$  iteratively, starting with $i=k
$. On $\x_k$ we construct sections $\sigma_0^{(k)},\ldots,\sigma_{N_k}^{(k)}$ as above, which restricted to 
$\x_{k,\delta_{k}}$ provides a basis of sections of $\tilde{L}^{(k-1)}$ on $\blX_{k-2}$ vanishing with
multiplicity $\geq m_{k-1}$ in $P_{k-1}$. We can complete these sections to a basis of all sections of $\tilde{L}^{(k-2)}$ on
$\blX_{k-2}$ by adding sections which vanish with multiplicity $< m_{k-1}$ in $P_{k-1}$. This basis can be used to construct 
trivializing sections of $\tilde{\mathcal{L}}^{(i)}$ on $\x_{k-1}$ as above, because by its construction it can be split up into the 
subsets $B_0, B_i,\tilde{B}_i$ and $B^{(i)}$. The K\"ahler form on $\mathcal{E}^{(i)}$is also as requested since we can 
choose the completing basis sections of $\tilde{L}^{(k-2)}$  as required above. Iterating this process for $\x_{k-2},
\ldots,\x_{1}$ we deduce property (3) for each $i=1\ldots,k$.  
\end{proof}

\section{K\"ahler packings of projective, complex, manifolds.}

\begin{thm}\label{thmB}
Let $X$ be a projective complex manifold of dimension $n$. $L$ an ample line bundle on $X$ and $P_1,
\ldots, P_k \in X$. Let $\epsilon_0= \epsilon(X;L,P_1,\ldots,P_k)$ denote the multipoint Seshadri 
constant of $L$ on $X$ in $P_1,\ldots,P_k$. Then, for any radius $r < \sqrt{\epsilon_0}$ there exists 
a K\"ahler packing of $k$ flat K\"ahler balls of radius $r$ into $X$. 

\end{thm}
\begin{proof}
First we construct families $\x_i$ as in the previous section. Then we embed Fubini-Study 
K\"ahler balls of large enough volume on $\mathcal{E}^{(i)}_i -E_i$ provided with the K\"ahler
metric $\omega_i^{(i)}$ induced by the global sections of $\Li$ constructed in Theorem \ref{thmA}  
on $\mathcal{E}_i^{(i)}$, 
for $i= k,k-1,\ldots,1$. These balls can be deformed to 
K\"ahler balls on non-central fibers $\x_{i,\delta_i} \cong \blX_{i-1}$, and then iteratively to 
non-central fibers $\x_{j,\delta_j} \cong \blX_{j-1}$ for $j=i-1,\ldots,1$ if the $\delta_j$ are chosen 
small enough. Doing this carefully the deformed balls will not intersect on $\x_{i,\delta_i} \cong X$, so after gluing in standard K\"ahler balls into the Fubini-Study K\"ahler balls we obtain the claim. 
 In more details:
\begin{enumerate}
\item Let $\Delta_{\delta} \subset \A^1_{\C}$ denote the open disk of radius $\delta$, with affine parameter $t$.\ Choose local coordinates $y_1,\ldots,y_n$ of $X$ around $P_i$, so $t,y_1,\ldots,y_n$
are local coordinates around $(0,P_i)$ in $\x_i$. Then over the open subset $\mathcal{U}_t \subset \x_i$ where these coordinates are defined there is a chart of the blow up of $\mathcal{X}_i$ in $(P_i,0)$ with coordinates $t,z_1,\ldots,z_n$ such that the blow up map to $\mathcal{U}_t$ is described by  
$(t,z_1,\ldots,z_n) \mapsto (t,y_1,\ldots,y_n) \coloneqq (t,tz_1,\ldots,tz_n)$. 
Because the coordinates $y_i$ are bounded the central fiber of the induced projection of 
$\mathcal{U}_t$ onto 
$\A^1_{\C}$ is $\mathcal{E}-E \cong \A^n_{\C}$. The non-central fibers are not quite  isomorphic to 
$\A^n_{\C}$  but contain balls $B_R(0)$ with $R$ arbitrarily large close to $t=0$. 
Thus for $R$ arbitrarily large we can find $\delta$ sufficently small and an embedding 
$\iota \colon \Delta_{\delta} \times B_R(0) \hookrightarrow \x_i$. See Figure \ref{product}. \\
\item Choosing in (1) $R$ large enough and $\delta$ small enough implies that 
 there exists embeddings of Fubini-Study K\"ahler balls $B_{R'}(0)$  with $ R' \leq R$, of 
volume arbitrarily close to the volume of $(\mathcal{E}^{(i)}_i, \omega_i^{(i)})$ in all
fibers of $ \Delta_{\delta} \times B_R(0)$ over $t \in \Delta_{\delta}$ with respect 
to the same Fubini-Study K\"ahler form $\omega_i^{(i)}$. \\
\item  By continuity, for $t $ small enough these Fubini-Study forms $\omega_i^{(i)}$
 differ to an arbitrarily small amount from 
the K\"ahler form $\omega_{i,t}$ on $\x_{i,t}$ obtained from the trivializing sections of $\Li$
 pulled back via the embedding $\iota_i$. This allows us to glue in the Fubini-Study 
K\"ahler balls of step (2) into non-central fibers $\x_{i,t}$ provided with the K\"ahler form 
$\omega_{i,t}$ for $t \ll 1 $ small enough.
Assume that $\omega_i^{(i)} = \frac{i}{2\pi}\partial\bar{\partial}\log (s_i^{(i)})$ and 
$\omega_{i,t}= \frac{i}{2\pi}\partial\bar{\partial} \log s_{i,t}$ on $ \{t\} \times B_R(0)$, 
where $s_i^{(i)}, s_{i,t}$ are functions constructed in the usual way from the appropriate sections.
Then choose a partition of unity $(\rho_1,\rho_2)$ such that $\rho_2 \restrict{B_{R'}(0)} \equiv 1$ and 
$\rho_2 \restrict{B_{R'}(0)} \equiv 0$. The glued 2-form $\tilde{\omega}_{i,t}$ is given by 
$\frac{i}{2\pi}\partial\bar{\partial} \log (\rho_1 s_i^{(i)}+\rho_2 s_{i,t})=
\frac{i}{2\pi}\partial\bar{\partial} \log( \rho_1(s_i^{(i)}-s_{i,t})+s_{i,t})$. This form is obviously closed, and it is non-degenerate because $s_i^{(i)}-s_{i,t}$ gets arbitrarily small
for $t \ll 0 $, thus $\rho_1(s_i^{(i)}-s_{i,t})+s_{i,t}$ is arbitrarily close to $s_{i,t}$. Consequently 
$\tilde{\omega}_{i,t}$ is a K\"ahler form. \\

\item The $i^{th}$ Fubini-Study K\"ahler ball on $\x_{i,\delta_i}$ does not intersect 
the $(i+1)^{st},\ldots, k^{th}$ K\"ahler ball constructed before: On the central
fiber $\x_{i,0}$ these balls lie on $\mathcal{E}_i^{(i)}-E_i$ and 
 $\blX_i -E_i \cong \blX_{i-1}-P_i$, so do not intersect. This will not change
 when we deform the balls to $\x_{i,\delta_i}$ if we choose 
$\delta_i$ small enough.\\ 

\item Assume $\int \omega^n_{FS} =1 $, and
$\int_{B_1(0)}\omega_{std}=1$. Then there exists a K\"ahler embedding 
$(B_r(0), \omega_{std}) 
\hookrightarrow (\C\mathbb{P}^n, \omega_{FS})$ , for all $r<1$ (for more details on this embedding see
\cite{Eckl}). Hence for all $r<\sqrt{m_i}$ there exists a 
K\"ahler embedding $(B_r(0), \omega_{std}) \hookrightarrow (\C\mathbb{P}^n, m_i\omega_{FS})$. Rescaling by $d_i$ we obtain a K\"ahler embedding $(B_r(0),\omega_{std}) \hookrightarrow (\C\mathbb{P}^n, 
\frac{m_i}{d_i}\omega_{FS})$
  for all $r< \sqrt{\frac{m_i}{d_i}}$. Since $\omega_i^{(i)}=m_i \omega_{FS}$ the embeddings constructed 
  above imply that 
  $\frac{m_i}{d_i}<\epsilon_0$, but since $\frac{m_i}{d_i}$ can be chosen arbitrarily close to 
  $\epsilon_0$ we can conclude that $\sqrt{\epsilon_0} \leq \gamma_k$, where $\gamma_k$ is the $k-$ball packing constant.

\end{enumerate} 

\end{proof}

\begin{figure}[H]
\centering

\begin{tikzpicture}[scale=0.7]
\draw (0,-3) to (0,3);
\draw (-2,-2) node (v3) {} to (-2,2) node (v1) {};
\draw (2,-2) node (v4) {} to (2,2) node (v2) {};

\draw[dashed]  plot[smooth, tension=.7] coordinates {(v1) (0,2.5) (v2)};
\draw[dashed]  plot[smooth, tension=.7] coordinates {(v3) (0,-2.5) (v4)};
\draw (-2,0) -- (2,0);

\draw  plot[smooth, tension=.7] coordinates {(-2,1.5) (0,1.9) (2,1.5)};
\draw  plot[smooth, tension=.7] coordinates {(-2,1) (0,1.3) (2,1)};
\draw  plot[smooth, tension=.7] coordinates {(-2,0.5) (0,0.65) (2,0.5)};
\draw  plot[smooth, tension=.7] coordinates {(-2,-0.5) (2,-0.5)};
\draw  plot[smooth, tension=.7] coordinates {(-2,-1) (0,-1.3) (2,-1)};
\draw  plot[smooth, tension=.7] coordinates {(-2,-1.5) (0,-1.9) (2,-1.5)};
\draw (-2.5,-4) to (2.5,-4);
\draw[->] (0,-3.2) to (0,-3.6);
\node at (3,-3.8){$\mathbb{A}^1_{\mathbb{C}}$};
\node at (2.9,0){$B_R(0)$};
\node at (0.2,3.2){$\mathbb{A}^n_{\mathbb{C}}$};

\draw  plot[smooth, tension=.7] coordinates {(-0.5,-2.45) (-0.5,2.45)};
\draw  plot[smooth, tension=.7] coordinates {(-1,2.3) (-1,-2.3)};
\draw  plot[smooth, tension=.7] coordinates {(-1.5,2.2) (-1.5,-2.2)};
\draw  plot[smooth, tension=.7] coordinates {(0.5,2.45) (0.5,-2.45)};
\draw  plot[smooth, tension=.7] coordinates {(1,2.3) (1,-2.3)};
\draw  plot[smooth, tension=.7] coordinates {(1.5,2.2) (1.5,-2.2)};
\end{tikzpicture}
\caption{}
\label{product}
\end{figure}

\begin{proof}[Proof of Theorem \ref{Main}]
 The claim that the k-point Seshadri constant is less or equal to the K\"ahler packing constant is 
a direct consequence of Theorem \ref{thmB}. The converse argument is a consequence of the symplectic blow up construction \cite{Polt}. The embedded balls allow us to construct K\"ahler forms  on the blow up 
$\pi$ of the centres whose curvature lies in the first Chern class of $\pi^*L -\sum\gamma E_i$.
\end{proof}

\begin{rem}
The method to prove Theorem \ref{thmB} also allows us to construct K\"ahler packings of balls with radius 
$r_1,\ldots,r_k$ arbitrarily close to $\epsilon_1,\ldots,\epsilon_k$ as long 
as $\pi^*L-\sum_{i=1}^k\epsilon_iE_i$ is nef. 
\end{rem}
As a final remark in \cite{Eckl} it was shown that when $X$ is a toric variety 
and $L$ a toric invariant divisor, the sections that generate the K\"ahler
form also induce a moment map whose image is the well known toric polytope associated to $X$ and $L$. Furthermore we find that the cut off triangles of the polytope are the shadows under the moment map of the glued in balls. In future work I will show that we can generalise this somewhat by replacing the moment polytopes by 
Newton-Okounkov bodies and a non-toric moment map. To illustrate these ideas I will construct examples of $\mathbb{P}^2$ blown up at 2 and 3 points with new shadows.

\bibliographystyle{alpha}
\bibliography{refs}

\end{document}